\documentclass[11pt]{amsart}

\usepackage{amscd}
\usepackage{amsmath,amsfonts,amssymb,amsthm}
\newcommand{\cal}{\mathcal}
\makeindex
\usepackage{color}
\usepackage{xcolor}
\usepackage[utf8]{inputenc}
\ifx\pdfpagewidth\undefined 
\usepackage[T1]{fontenc}
\else 
\usepackage[OT1]{fontenc} 
\fi 
\usepackage{amssymb,amsfonts}
\usepackage{eurosym}
\usepackage{vmargin}
\setmarginsrb{21mm}{7mm}{21mm}{11.5mm}%
        {21mm}{5mm}{21mm}{11.5mm}
\makeindex

\newtheorem{deff}{Definition}[section]
\newtheorem{lemma}[deff]{Lemma}

\newtheorem{theorem}[deff]{Theorem}

\newtheorem{corollary}[deff]{Corollary}

\newtheorem{proposition}[deff]{Proposition}

\newtheorem{notation}[deff]{Notation}

\newtheorem{fact}[deff]{Fact}

\newtheorem{em-Beispiel}[deff]{\textcolor{red}{Beispiel}}
\newtheorem{em-beispiel}[deff]{Beispiel}
\newtheorem{em-beispiele}[deff]{Beispiele}
\newtheorem{em-def}[deff]{Definition}        
 \newtheorem{em-Def}[deff]{\textcolor{green}{Definition}}
\newtheorem{em-remark}[deff]{Remark}         
\newtheorem{em-bemerkungen}[deff]{Bemerkungen}
\newtheorem{em-Bemerkung}[deff]{\textcolor{red}{Bemerkung}}
\newtheorem{em-question}[deff]{Question}
\newtheorem{em-questions}[deff]{Questions}

\newtheorem{problem}[deff]{Problem}

\theoremstyle{definition}
\newtheorem{example}[deff]{Example}

\newcommand{\NB}{$\clubsuit$}
\newcommand{\NL}{$\spadesuit$}

\newcommand\uu{{\mathbf u}}

\newcommand\vv{{\mathbf v}}

\definecolor{agbgreen}{rgb}{0.0, 0.5, 0.0}

\renewcommand{\phi}{\varphi}

\newenvironment{remark}{\begin{em-remark} \em }{\end{em-remark}}

\newenvironment{question}{\begin{em-question}\em }{\end{em-question}}
\newenvironment{questions}{\begin{em-questions}\em }{\end{em-questions}}

\renewcommand{\P}{\mathbb P}

\newcommand{\R}{\mathbb R}

\newcommand{\Z}{\mathbb Z}
\newcommand{\T}{\mathbb T}
\newcommand{\N}{\mathbb N}

\renewcommand{\phi}{\varphi}
\renewcommand{\theta}{\vartheta}

\newcommand{\FF}{\mathcal F}
\newcommand{\I}{\mathcal I}
\newcommand{\VV}{\mathcal V}

\numberwithin{equation}{section}

\title{The Protasov-Zelenyuk topology and ideal convergence}

\author{Lydia Au\ss{}enhofer}
\address{Universit\"at Passau, Innstr. 33, D-94032 Passau, ORCID: 0000-0003-1269-4735}
\email{lydia.aussenhofer@uni-passau.de}

\author{Dikran Dikranjan}
\address{Universit\`a di Udine, DMIF, Via delle Scienze 206, 33100 Udine, ORCID: 0000-0002-1159-9958}
\email{dikran.dikranjan@uniud.it}

\author{Anna Giordano Bruno}
\address{Universit\`a di Udine, DMIF, Via delle Scienze 206, 33100 Udine, ORCID: 0000-0003-3431-2240}
\email{anna.giordanobruno@uniud.it}

\date{\today}
\newcommand{\dis}{\displaystyle}
\begin{document}

\maketitle

\begin{abstract}
The so-called $T$-sequences $\uu$ in a group $G$, and the related finest Hausdorff group topology $T_\uu$ on $G$ that makes $\uu$ a null sequence, were introduced by Protasov and Zelenyuk 35 years ago and since then they became a fundamental tool in the field of topological groups.

More recently, in the abelian case, the subfamily of $T$-sequences called $TB$-sequences was introduced, as well as the finest precompact group topology $T_{\mathbf{pu}}$ that makes $\uu$ a null sequence.

Here we study the counterpart of all these notions with respect to ideal convergence in place of the classical notion of convergence of a sequence.
Also, we study their relation to the already established field of ${\cal I}$-characterized subgroups of compact abelian groups.
\end{abstract}

\medskip\medskip

{\footnotesize
\noindent  {\sf Mathematics Subject Classification 2020:}
   11B05, 
  20K45, 22A05, 22C05,
  54H11 

\noindent  {\sf Keywords:} ideal convergence, $T$-sequence, $TB$-sequence, characterized subgroup, $\I$-characterized subgroup}
\bigskip

\maketitle

\tableofcontents



\section{Introduction}

Let $G$ be an abelian group. In the sequel we write a single boldface character $\uu$ to denote a sequence $\uu = (u_n)_{n\in\N_0}$
when no danger of confusion exists.

\subsection{$T$- and $TB$-sequences}\hfill

We start recalling the following fundamental notions.

\begin{deff}\label{Tdef}
A sequence $\uu$ in an abelian group $G$ is called a {\bf $T$-sequence} (resp., {\bf $TB$-sequence}), if there is a Hausdorff (resp., precompact) group topology $\tau$ on $G$ such that $u_n{\longrightarrow}0$ in $(G,\tau)$.
\end{deff}

If $\uu$ is a $T$-sequence (resp., $TB$-sequence) in an abelian group $G$, then there exists a finest Hausdorff group topology $T_{{\bf u}}$ (resp., a finest precompact group topology $T_{\mathbf p\uu}$) on $G$ in which $\uu$ converges to $0$, namely, the supremum of all Hausdorff (resp., precompact) group topologies with this property. Clearly, $T_{\mathbf p\uu}\subseteq T_\uu$ for any $TB$-sequence $\uu$.

The correspondence $\uu \mapsto T_\uu$ is monotone decreasing, i.e., if ${\bf u}$ is a $T$-sequence and ${\bf v}\sqsubseteq {\bf u}$ is a subsequence of ${\bf u}$, then also ${\bf v}$ is a $T$-sequence and $T_{{\bf v}}\supseteq T_{{\bf u}}$. The same holds for $\uu\mapsto T_{\mathbf p\uu}$.

For an abelian group $G$ and a subgroup $H$ of $\mathrm{Hom}(G,\T)$, denote by $\sigma(G,H)$ the topology on $G$ generated by the characters of $H$, which is totally bounded. Then $H=(G,\sigma(G,H))^\wedge$~\cite[Proposition~11.3.10]{ITG}.
One can prove that (e.g., see~\cite{Fusion})
\begin{equation}\label{wedgeq}
T_{\mathbf p\uu}=\sigma(G,(G,T_{\mathbf p\uu})^\wedge)=\sigma(G,(G,T_{\uu})^\wedge).
\end{equation}

Since every precompact group topology is Hausdorff, every $TB$-sequence is a $T$-sequence. In~\cite{PZ} (cf.~also~\cite[14.4.18]{ITG}), an example of a $T$-sequence which is not a $TB$-sequence is given.
In particular, $\uu$ is not a $TB$-sequence, when $(G, T_{\uu})$ is minimally almost periodic.
In~\cite{Fusion} we present a new class of $T$-sequences $\uu$ in $\Z$ for which the group topology $T_\uu$ is minimally almost periodic, so $\uu$ cannot be a $TB$-sequence.

\begin{theorem}\label{ExTs}{\rm~\cite[Example 1.13(d)]{DGT}}
In every infinite abelian group there exists a non-trivial $TB$-sequence.
\end{theorem}

It was proved earlier that every infinite group admits a non-trivial $T$-sequence~\cite[Theorem~2.1.7]{PZ}.

\subsection{$T^\I$- and $TB^\I$-sequences}\hfill

Recall that an ideal $\I$ on a set $X$ is a subset of the powerset of $X$ which is closed under taking subsets and finite unions. We are mainly interested in ideals on $\N_0$. Clearly, for every $A\subseteq \N_0$, the powerset of $A$ is an ideal on $\N_0$. Another prominent example is the ideal $\FF in$ on $\N_0$ of finite subsets of $\N_0$.

Next we recall the notion of ideal convergence:

\begin{deff}\cite{Cartan}
Let ${\cal I}$ be an ideal on $\N_0$, let $(X,\tau)$ be a topological space and $(x_n)$ a sequence in $X$.
We say that $(x_n)$ \emph{$\I$-converges} to $x$, denoted by $x_n\stackrel{{\cal I}}{\longrightarrow}x$, whenever $\{n\in \N_0:\ x_n\notin U \} \in  \cal I$ for every neighborhood $U$ of $x $ in $(X,\tau)$.
\end{deff}

Clearly, $\mathcal F in$-convergence is the usual convergence, so if $\mathcal F in\subseteq\I$, then $x_n\to x$ implies $x_n\stackrel{{\cal I}}{\longrightarrow}x$.

By analogy with Definition~\ref{Tdef}, we introduce the following extended notion.

\begin{deff}
Let ${\cal I}$ be an  
ideal on $\N_0$ and let $\uu$ be a sequence in an abelian group $G$.
If there is a Hausdorff (resp., precompact) group topology $\tau$ on $G$ such that $u_n\stackrel{{\cal I}}{\longrightarrow}0$ in $(G,\tau)$, we call $\uu$ a  {\bf  $T^\I$-sequence} (resp., {\bf  $TB^\I$-sequence}).
\end{deff}

Of course, every $TB^\I$-sequence is a $T^\I$-sequence, but the converse is not true as we mentioned above for $\I=\mathcal Fin$.
Nevertheless, one can try to find ideals as in the following problem, where one should start with the case $G=\Z$.

\begin{problem}
Let $G$ be an abelian group.
Characterize those ideals ${\cal I}$ on $\N_0$ for which every $T^\I$-sequence in $G$ is also a $TB^\I$-sequence.
\end{problem}

To avoid trivialities (e.g., see Example~\ref{tr} and~\cite{DRAH}), we usually impose the ideal $\I$ to be {\bf free} (i.e., $\mathcal F in\subseteq\I$) and {\bf proper} (i.e. $\I\not=\mathcal{P}(\N_0)$).

\smallskip
Clearly, $T^{\mathcal Fin}$-sequences are exactly $T$-sequences, and $TB^{\mathcal Fin}$-sequences are exactly $TB$-sequences.
Moreover, in case $\I$ is a free ideal, every $T$-sequence is a $T^\I$-sequence, and every $TB$-sequence is a $TB^\I$-sequence.


If ${\bf u}$ is a $T^\I$-sequence (resp., $TB^\I$-sequence) in an abelian group $G$, then there exists a finest Hausdorff group topology  $T^{{\cal I}}_{{\bf u}}(G)$ (resp., a finest precompact group topology $T^{\I}_{\mathbf p\uu}(G)$) on $G$ in which $\uu$ $\I$-converges to $0$. (Take the supremum of all Hausdorff (resp., precompact) group topologies with this property and observe that it has again the property $u_n\stackrel{{\cal I}}{\longrightarrow}0$.)

 When there is no possibility of confusion, we simply write $T^{{\cal I}}_{{\bf u}}$ in place of $T^{{\cal I}}_{{\bf u}}(G)$, respectively
 $T^{{\cal I}}_{\mathbf p{\bf u}}$ in place of $T^{{\cal I}}_{\mathbf p{\bf u}}(G)$. Clearly, $T^{{\cal I}}_{\mathbf p{\bf u}}\subseteq T^{{\cal I}}_{{\bf u}}$ for any $TB^\I$-sequence $\uu$.
The correspondence $\uu \mapsto T^\I_\uu$ is monotone decreasing, i.e., if ${\bf u}$ is a $T^\I$-sequence and 
${\bf v}\sqsubseteq {\bf u}$, then also ${\bf v}$ is a $T^\I$-sequence and $T_{{\bf v}}^\I\supseteq T_{{\bf u}}^\I$. The same holds for $\uu\mapsto T^\I_{\mathbf p\uu}$.

\smallskip
As consequence of the results in the case $\I=\mathcal Fin$ (more specifically, Theorem \ref{ExTs}), we immediately get that every infinite abelian group has a $TB^\I$-sequence for every free ideal $\I$ on $\N_0$. Moreover, we establish the existence of a metrizable group topology that makes a $T^\I$-sequence converge to $0$ (see \S\ref{first}).
Furthermore, in Section~\ref{explicitsec} we find an explicit description of a neighborhood base at $0$ in $T_\uu^\I$ (see Proposition~\ref{explicit}), extending its counterpart for the Protasov-Zelenyuk topology established in~\cite{PZ} (see also~\cite{Schar} and~\cite[\S 5.3.3]{ITG}).

In Section~\ref{sub}, some information concerning the topology $(G,T^\I_{\uu})$ is provided.  In case there exists $I\in \I$ such that $(u_n)_{n\in\N\setminus I}$ is a $T$-sequence, then the whole sequence $\I$-converges to $0$ and so it is a $T^\I$-sequence. As a consequence, the topology $T^\I_{\uu}$ is finer than $T_{(u_n)_{n\in\N_0\setminus I}}$ (see Proposition~\ref{221}). This is a very helpful observation in the case of $P$-ideals.

As a consequence, many compact subsets of $(G,T^\I_{\uu})$ can be found. Although, see Question~\ref{Ques1}, we do not know if all compact subsets can be  found in this way.

\smallskip
In Section~\ref{susec} we study $\I$-characterized subgroups and their   relation with $T^\I$- and $TB^\I$-sequences. Following~\cite[Definition 7.25]{DRAH}, for an ideal $\I$ on $\N_0$, a topological abelian group $G$, and a sequence $\uu$ in $G^\wedge$, let ${s}^{{\cal I}}_{{\bf u}}(G)= \{x\in G: u_n(x) \stackrel{{\cal I}}{\longrightarrow} 0\}.$
A subgroup $H$ of $G$ is an {\bf ${\cal I}$-characterized subgroup} of $G$ if $H=s_\uu^\I(G)$ for some sequence $\uu$ in $G^\wedge$.

Clearly, $s_{{\bf u}}(G)={s}^{{\cal F}in}_{{\bf u}}(G)= \{x\in G: u_n(x) \to 0\}$ is the subgroup of $G$ characterized by $\uu$ (see~\cite{BDS,MAP} and the survey~\cite{DDG}).
The first instance of novelty in the theory of characterized subgroups was given by the ideal $\I_d$ of subsets of $\N$ with zero natural density (see Example~\ref{Id}). Indeed, the subgroups ${s}^{{\cal I}_d}_{{\bf u}}(\T)$ of $\T$ were very briefly mentioned by Borel~\cite{BO} in completely different terms; unaware of this, Bose, Das and the second named author re-introduced them recently in~\cite{DDB} under the name {\bf statistically characterized}. The latter paper was the first of a long list of papers on statistically characterized subgroups of $\T$ and more extended notions, among which the most general one is that of $\I$-characterized subgroups for an ideal $\I$ on $\N_0$ (see the survey~\cite{DRAH}).

It turns out that algebraically, the $\I$-characterized subgroup $s_{\uu}^\I(G)$ can be canonically identified with  the character group of $(G^\wedge, T^\I_{\uu})$  (see Theorem~\ref{cg}).
  Moreover, we extend \eqref{wedgeq} in Corollary~\ref{1.1I} to this setting.

It was proved by Comfort, Trigos and Ta Sun Wu~\cite[Proof of Lemma~3.10]{CJW} that when $G$ is a compact abelian group, ${s}_{{\bf u}}(G)$ is always an $F_{\sigma\delta}$-set of $G$.  Here (see Theorem~\ref{sd}) we prove analogously that $s_{\bf u}^\I(G)$ is an $F_{\sigma\delta}$-set of $G$ for any compact abelian group $G$ and any analytic $P$-ideal $\I$ on $\N_0$.
This was proved in~\cite{DDB} in the special case $\I=\I_d$ and $G=\T$, that is, for statistically characterized subgroups of $\T$.
Since an $\I$-characterized subgroup $H$ of a compact abelian group $G$ is a Borel set by Theorem~\ref{sd}, its Haar measure $\mu(H)$
is well defined; it is $0$ if $H$ is of infinite index in $G$, otherwise $\mu(H) = 1/[G:H]$. Hence, $\mu(H) =0$ when $G$ is connected   and $H$ is proper.

\subsection*{Notation and terminology}\hfill

We denote by ${\N}$ the sets of naturals $\{1,2,3, \dots\}$ and put $\N_0 = \{0\} \cup \N$. The set of primes is denoted by $\P$.

We consider a sequence $\uu$ in an abelian group $G$ also as a function $\N_0\to G$.
In this vein, for an infinite subset $J\subseteq \N_0$ we shall simply write only $\uu\restriction_J$ for the (sub)sequence $(u_n)_{n\in J}$. We write ${\bf v}\sqsubseteq_J{\bf u}$ if $(v_k)$ is the subsequence $(u_n)_{n\in J}$ and ${\bf v}\sqsubseteq {\bf u}$ to just say that ${\bf v}$ is a subsequence of ${\bf u}$.

In a topological space $(X,\tau)$, denote by $\VV_\tau(x)$ the filter of neighborhoods of $x$. If $X$ is a topological group with neutral element $e$, then we briefly write $\VV_\tau$ in place of $\VV_\tau(e)$.

For a topological abelian group $G$ we denote by $G^\wedge$ its Pontryagin dual.

  We denote $\T=\R/\Z$ endowed with the topology induced by the topology of $\R$. Let $\varphi\colon\R\to \T$ be the canonical projection. For every $k\in\N$, let $\T_k=\varphi([-\frac{1}{4k},\frac{1}{4k}])$.

\section{Background on $T$- and $TB$-sequences}


It is of course of great importance to have  a method to check whether a given sequence is a $T$-sequence or not.
The following criterion  (Theorem~\ref{characterization_T}) is shown in~\cite{PZ}. 

\begin{notation}\label{um}
For a sequence $\uu$  in $G$ and $m\in\N_0$, put $$\uu_m=\{0\}\cup\{\pm u_k:k\geq m\}.$$
For $k\in\N$, an abelian group $G$ and $B\subseteq G$, put
$$k\odot B = \underbrace{   (\{0\}\cup\pm  B)+     (\{0\}\cup \pm B) + \ldots +    (\{0\}\cup \pm B)}_{k\ \mbox{\tiny summands}}$$
\end{notation}

\begin{theorem}\cite[2.1.4]{PZ}
\label{characterization_T} 
A sequence $\uu$ in an abelian group $G$ is a $T$-sequence if and only if for every $g\in G\setminus \{0\}$ and every $k\in\N$ there exists $m\in\N$ such that $g\notin k\odot \uu_m$, equivalently, $\bigcap_{m\in\N}k\odot\uu_m=\{0\}$ for every $k\in\N$.
\end{theorem}

The property of being a $T$-sequence (resp., $TB$-sequence) is intrinsic for the subgroup $H$ generated by the sequence $\uu$:

\begin{lemma}\label{gen}
Let $H$ be an abelian group.
If $\uu$ is a $T$-sequence (resp., $TB$-sequence) in $H$, then it is also a $T$-sequence (resp., $TB$-sequence)
in any abelian group $G$ containing $H$.
\end{lemma}
\begin{proof}
Let $G$ be an abelian group containing $H$. For $T$-sequences the assertion is trivial, just declare the subgroup $H$ to be open in $G$.

In case $\uu$ is a $TB$-sequence, consider the surjective restriction homomorphism $R\colon {\rm Hom}(G,\T)\to {\rm Hom}(H,\T)$ such that $R(\chi)=\chi\restriction_H$ for every $\chi\in\mathrm{Hom}(G,\T)$. Let $\sigma$ 
be a precompact group topology on $H$ such that $\uu$ converges to $0$ in $\sigma$, and let  $(H,\sigma)^\wedge= D\le {\rm Hom}(H,\T)$. So $\sigma=\sigma(H,D)$. Then $\widetilde\sigma=\sigma(G,R^{-1}(D))$ is a totally bounded group topology on $G$ inducing $\sigma$ on $H$. Obviously, every element $h\in H\setminus \{0\}$ can be separated from $0$ using characters in $(G,\widetilde\sigma)^{\wedge}$; since $H^\bot\le R^{-1}(D)$, also every element in $G\setminus H$ can be separated from $0$. Hence, $\widetilde\sigma$ is Hausdorff, so precompact, and $\uu$ converges to $0$ in $(G,\widetilde \sigma)$.
\end{proof}

In view of this result, in the sequel we shall assume that $\uu$ generates $G$.

\begin{remark}\label{hemicpt}\cite[Lemma 2]{Schar}\cite[4.1.5]{PZ}
If $\uu$ is a $T$-sequence in an abelian group $G$ such that $\uu$ generates $G$, then $(G,T_{\uu})$ is a hemicompact $k$-space, more precisely, $G=\bigcup_{m\in\N}m\odot \uu$ and $A\subseteq G$ is closed if and only if $A\cap (m\odot \uu)$ is compact for every $m\in\N$.
Moreover, if $C\subseteq G$ is compact, then $C\subseteq m\odot\uu$ for a suitable $m\in\N$.
\end{remark}

From this remark it follows easily that  the topology $T_{\uu}$ is not metrizable if $\uu$ is a one-to-one sequence. Anyway, there exists a (strictly) coarser metrizable group topology in  which $\uu$ converges to $0$: 


\begin{proposition}\cite{ClassIT}\label{Met}
Let $G$ be an abelian group.
\begin{itemize}
\item[(a)] If $\mathcal T$ is a group topology on $G$ such that $(G,{\cal T})$ has an open countable subgroup, then there exists a metrizable group topology $\tau$ on $G$ coarser than ${\cal T}$.
\item[(b)] Consequently, if $\uu$ is a $T$-sequence in $G$, then there is a metrizable group topology $\tau$ (strictly coarser than $T_{\uu}$) on $G$ such that $u_n{\longrightarrow}0$ in $(G,\tau)$.
\end{itemize}
\end{proposition}

Item (b) follows from (a) since $\langle u_n: n\in\N_0\rangle$ is open in $T_{{\bf u}}$.

An analog of the above proposition does not hold for $TB$-sequences due to the following restriction:

\begin{proposition}\cite{ClassIT}\label{Metb}
Let $G$ be an abelian group. Then:
\begin{itemize}
\item[(a)] $G$ admits a precompact metrizable group topology if and only if $|G|\le \mathfrak{c}$.
\item[(b)] if $|G|\le  \mathfrak{c}$ and $\uu$ is a $TB$-sequence in $G$, 
then there is a precompact metrizable group topology $\tau$ on $G$ such that $u_n{\longrightarrow}0$ in $(G,\tau)$.
\end{itemize}
\end{proposition}
 Recall that an abelian group $G$ is called {\bf almost torsion free} if for every prime $p$ the $p$-rank $r_p(G)$ is finite.
The following theorem is one of the main results of~\cite{ClassIT}:

\begin{theorem}~\cite{ClassIT} \label{classITth}
For an abelian group $G$ the following assertions are equivalent:
\begin{enumerate}
\item[(a)] $G$ is almost torsion free or of prime exponent;
\item[(b)] every one-to-one sequence in $G$ has a $T$-subsequence;
\item[(c)] every one-to-one sequence in $G$ has a $TB$-subsequence.
\end{enumerate}
\end{theorem}



%


\section{Ideal convergence and ideal $T$- and $TB$-sequences}

\subsection{Ideals and ideal convergence}\hfill

If we impose no restrictions on the ideal, the definition of ideal convergence recalled in the introduction may give some unpleasant consequences as the following example shows.

\begin{example} \label{tr}
If $\I$ is the maximal ideal on $\N_0$ generated by the set $\N$ of all positive naturals (i.e., $I\in \I$ if and only if $0 \not \in I$), then  $u_n\stackrel{{\cal I}}{\longrightarrow}0$ if and only  $u_0 \in U$ for every $U \in \VV_\tau$,
all the other members of the sequence can be chosen arbitrarily. In case  $(G,\tau)$ is Hausdorff, this occurs precisely when $u_0=0$. Hence, if $(G,\tau)$ is Hausdorff, $u_n\stackrel{{\cal I}}{\longrightarrow}0$ if and only if $u_0=0$ (regardless of the rest of the sequence).
 \end{example}

%

If $(x_n)$ is a convergent sequence in a topological space, then every reordering $(x_{\pi(n)})$ where $\pi$ denotes a permutation of $\N_0$ is also convergent.
The next  example shows that in case of ideal convergence, this property fails as strongly as possible.

\begin{example}
Let $\pi$ be a permutation of $\N_0$. Call an ideal $\I$ of $\N$ {\bf $\pi$-invariant}
 if for every $I\in  \I$, also $\pi (I)\in \I$ for all $I\in  \I$.

(a) Obviously, $\FF in$ is $\pi$-invariant for every permutation $\pi$ of $\N_0$. Conversely, if a proper free ideal $\I$ on $\N_0$ is $\pi$-invariant for every permutation $\pi$ of $\N_0$, then $\I = \FF in$. Indeed, if $\I\not=\FF in$ then there exists $I\in\I$ such that both $I$ and $\N_0\setminus I$ are infinite. Let $\pi$ be a permutation of $\N_0$ which maps $I$ to $\N_0\setminus I$ and vice versa.
 Then    $\pi^{-1}(I) = \pi(I)\not\in \I$, so $\I$ is not  $\pi$-invariant.

(b) If the ideal $\I$ on $\N_0$ is $\pi$-invariant and $x_n\stackrel{\I}{\longrightarrow}x$,  then $(x_{\pi^{-1}(n)})$ is also $\I$-convergent to $x\in X$. Indeed, let $W$ be a neighborhood of $x$. Then $\{n\in \N_0:\ u_{\pi^{-1}(n)}\notin W\}=\pi (\{m\in\N_0 :\ u_m\notin W\})\in \I$, since $\{m:\ u_m\notin W\}\in I$ and since $\I$ is invariant under $\pi$ by assumption. 

(c) If an ideal $\I$ on $\N_0$ is not $\pi$-invariant, then for every non-trivial Hausdorff topological group $G$, there exists a sequence $(x_n)$ in $G$ such that  $x_n\stackrel{\I}{\longrightarrow}0$ in $G$, but $x_{\pi^{-1}(n)}\stackrel{\I}{\not\longrightarrow}0$. To this end fix an element $0\ne g\in G$, a member $I \in \I$ such that $\pi (I)\not \in \I$ and define $x_n=g$ for $n\in I$ and $x_n=0$ for $n\notin I$. Then $(x_n)$ converges to $0$ in $G$, while $x_{\pi^{-1}(n)}\stackrel{\I}{\not\longrightarrow}0$ in any Hausdorff group topology on $G$ (i.e., $(x_{\pi^{-1}(n)})$ is not a $T^\I$-sequence).
\end{example}

Recall the definition of natural density:

\begin{example}\label{Id}
 For $A\subseteq\N_0$, denote $$A(n)=\{i\in A: i\leq n\}=A\cap[0,n].$$ The \emph{lower} and the \emph{upper natural density of $A$} are
$$\underline{d}(A)=\liminf_{n\to\infty}\frac{|A(n)|}{n}\quad\text{and}\quad\overline{d}(A)=\limsup_{n\to\infty}\frac{|A(n)|}{n}.$$
If $\underline{d}(A)=\overline{d}(A)$, we say that the \emph{natural density of $A$} exists and denote it by $d(A)$.

Then $\I_d=\{A\subseteq \N_0 : d(A)=0\}$ is an ideal on $\N_0$ that contains $\FF in$.
\end{example}

%

\begin{remark}\label{I_conv}
\begin{enumerate}
\item[(a)] Usual convergence is equivalent to $\FF in$-convergence.
\item[(b)] $\I_d$-convergence is usually referred to as \emph{statistical convergence}. 
\item[(c)] If $\I_1$ and $\I_2$ are ideal on $\N_0$ such that $\I_1\subseteq \I_2$, then $\I_1$-convergence implies $\I_2$-convergence, in particular, every convergent sequence is statistically convergent.
\end{enumerate}
\end{remark}


Recall that an ideal $\I$ on $\N_0$ is called a {\bf $P$-ideal} if for every sequence of elements $(I_n)$ in $\I$, there exists $I\in \I$ such that
$I\setminus I_n$ is finite for all $n\in\N$. The ideals $\mathcal Fin$ and $\I_d$ are free $P$-ideals.

 \begin{fact} \label{Remark}{~\cite[Proposition 3.2 and Theorem 3.2(ii)]{KSW}}
%
%
%
%
Let $\I$ be a proper free ideal on $\N_0$, let $(X,\tau)$ be a topological space, $x\in X$, and $(x_n)$ a sequence in $X$.
\begin{itemize}
\item[(a)] If for some $I\in \I$ the subsequence $(x_n)_{n\in\N_0\setminus I}$ converges to $x$, then the whole sequence $(x_n)$ ${\cal I}$-converges to $x$.
\item[(b)] If $\I$ is a $P$-ideal, $(X,\tau)$ is first countable and $(x_n)$ is $\I$-convergent to $x$, then there exists $I\in \I$ such that the subsequence $(u_n)_{n\in\N_0\setminus I}$ converges to $x$.
\end{itemize}
    \end{fact}
\begin{proof}
(a) To prove the first assertion note that for $W\in{\cal V}_\tau(x)$, the set $\{n\in\N_0:u_n\notin W\}\subseteq \{n\in\N_0\setminus I:\ u_n\notin W\}\cup I$ belongs to $ {\cal I}$, since the first member of the union is finite.

(b) Let $\{W_n:n\in\N\}$ be a base of the filter $\mathcal V_\tau(x)$.
For every $n\in\N$ the set of indices $I_n=\{n\in\N_0:u_n\notin W_n\}$ belongs to ${\cal I}$. Since ${\cal I}$ is a $P$-ideal, there exist $I\in {\cal I}$ and a sequence $(F_n)$ in $\mathcal Fin$ such that $I_n\subseteq I\cup F_n$ for all $n\in\N_0$.
But this implies that $(u_n)_{n\in\N_0\setminus I}$ converges to $x$.
\end{proof}
%

\begin{fact}\label{I-cont}
Let $f\colon X\to Y$ be a continuous mapping between topological spaces and $(x_n)$ a sequence in $X$. If $x_n\stackrel{{\cal I}}{\longrightarrow}x$, then $f(x_n)\stackrel{{\cal I}}{\longrightarrow}f(x)$.
\end{fact}

By using this fact, we find examples of sequences that are never $\I$-convergent (actually, none of their subsequences is $\I$-convergent):

\begin{proposition}
Let $G$ be an infinite abelian group which is neither almost torsion-free nor of prime exponent. Then there is a one-to-one sequence $\uu$ in $G$
such that no subsequence of  $\uu$ is $\I$-convergent for any group topology on $G$ and any proper ideal $\I$ on $\N_0$.
In particular, $\uu$ is not a $T^{\I}$-sequence for any proper ideal $\I$ on $\N_0$.
\end{proposition}
\begin{proof}
Let $\tau$ be a group topology on $G$.
By assumption,  there exists a prime $p$ such that $G[p]$ is infinite and $G[p]$ is a proper subgroup of $G$. Choose a one-to-one sequence $(a_n)$ in $G[p]$ and an element $b\in G\setminus G[p]$. Let $u_n=b+a_n$ for every $n\in\N_0$, and $\uu=(u_n)$. Then $p\uu=(pu_n)$, as well as any subsequence of $p\uu$, is the constant sequence $(pb)$. Since the multiplication mapping $G\to G,\ x\mapsto px$ is continuous, 
the sequence $\uu$ (as well as any subsequence of its subsequences) is not $\I$-convergent for any ideal $\I$ on $\N_0$.
\end{proof}

For example, let $\I$ be a proper ideal on $\N_0$ and $G = \Z/4\Z\times \Z/2\Z^{(\N_0)}$, then let $c=1+4\Z$ be a generator of $\Z/4\Z$ and $c_n$ a generator of the $n$-th copy of $\Z/2\Z$ for every $n\in\N_0$. The sequence $((c,c_n))$ is not a $T^{\I}$-sequence 
for any proper 
ideal on $\N_0$.


We modify this example to give a one-to-one sequence which is $\I$-convergent, but a suitable permutation of the sequence is no longer $\I$-convergent. 

\begin{example}
Let $\I$ be a proper ideal on $\N_0$. Then there exists $I\in\I$ such that both $I$ and $\N_0\setminus I$ are infinite. Let $\pi$ be a permutation of $\N_0$ which maps $I$ to $\N_0\setminus I$ (and vice versa).

With $G = \Z/4\Z\times \Z/2\Z^{(\N_0)}$ and $c=1+4\Z$ as above, define $u_n=(c,c_n)$ for $n\in I$ and $u_n=(0,c_n)$ for $n\notin I$. Then $(u_n)_{n\notin I}=(0,c_n)_{n\notin I}$ converges to $(0,0)$ in the product topology, in particular the whole sequence $\uu=(u_n)$ is $\I$-convergent to $0$ in the product topology by Fact~\ref{Remark}. So $\uu$ is a $T^\I$-sequence. 

We show now that  $(u_{\pi(n)})$ is not a $T^\I$-sequence. Indeed,
$$u_{\pi(n)}=\left\{\begin{array}{c@{\ :\ }c}(0,c_{\pi(n)})&\pi(n)\notin I\\ (c,c_{\pi(n)})&\pi(n)\in I \end{array}\right.=
\left\{\begin{array}{c@{\ :\ }c}(0,c_{\pi(n)})&n\in  I\\ (c,c_{\pi(n)})&n\notin I\end{array}\right..$$ If $(u_{\pi(n)})$ were a $T^\I$-sequence, then so would be $(2u_{\pi(n)})$. Choose a Hausdorff group topology $\tau$ in which $(2u_{\pi(n)})$ converges to $(0,0)$ and choose a neighborhood $W\in \VV_\tau$ that does not contain $(2c,0)$. Then $\{n\in\N_0:\ 2u_{\pi(n)}\notin W\}=\N_0\setminus I\notin \I$,
since $\I$ was assumed to be proper. This shows that $(u_n)$ is a $T^\I$-sequence, while $(u_{\pi(n)})$ fails to be a $T^\I$-sequence.
\end{example}

\subsection{First results on ideal $T$- and $TB$-sequences}\label{first}\hfill



%

As an immediate consequence of Remark~\ref{I_conv} and Theorem~\ref{ExTs} we have:

\begin{fact}
Let $\I$ be a free ideal on $\N_0$. Every infinite abelian group $G$ has a $TB^\I$-sequence.
\end{fact}

Next we find analogous results with respect to Propositions~\ref{Met} and~\ref{Metb}.

\begin{proposition}\label{MetStat*}
Let $\uu$ be a sequence in an abelian group $G$.
\begin{itemize}
\item[(a)] If ${\bf u}$ is a $T^\I$-sequence, then there exists a metrizable group topology $\tau$ on $G$, such that
$u_n\stackrel{{\cal I}}{\longrightarrow}0$ in $(G,\tau)$.
\item[(b)] If ${\bf u}$ is a $ TB^\I$-sequence in $G$ and $G=\langle \{u_n:n\in\N_0\}\rangle$ holds, then there exists a precompact metrizable group topology $\tau$ on $G$, such that
$u_n\stackrel{{\cal I}}{\longrightarrow}0$ in $(G,\tau)$.
\end{itemize}
\end{proposition}
\begin{proof}
(a) Apply Proposition~\ref{Met} to $(G,T^{{\cal I}}_{{\bf u}})$ and observe that the countable subgroup $\langle u_n: n\in\N_0\rangle$ is open in $T^{{\cal I}}_{{\bf u}}$.

(b) Apply Proposition~\ref{Met} to $(G,T^{{\cal I}}_{{\bf p}{\bf u}})$ and the open  countable subgroup $G$. Hence there exists a coarser metrizable group topology $\tau $ on $G$. Since $T^{{\cal I}}_{{\bf p}{\bf u}}$ is precompact, so is $\tau$. Since the identity $(G,T^{{\cal I}}_{{\bf p}{\bf u}})\to (G,\tau)$ is continuous, it follows from Fact~\ref{I-cont} that $u_n\stackrel{{\cal I}}{\longrightarrow}0$ in $(G,\tau)$.
%
\end{proof}

\section{Explicit description of a neighborhood base at $0$ in $T^{{\cal I}}_{{\bf u}}$}\label{explicitsec}

Let $\I$ be a free ideal on $\N_0$ and ${\bf u}$ be a $T^\I$-sequence in an abelian group $G$.
In this section we describe a neighborhood base at $0$ of $(G,T^{{\cal I}}_{{\bf u}})$.

\begin{remark} \label{maximality}
Let $\I$ be a free ideal on $\N_0$ and ${\bf u}$ be a sequence in an abelian group $G$. Since $\uu$ is $\I$-convergent to $0$ in the indiscrete  topology (which is a group topology), there always exists a finest group topology on $G$ in which $\uu$ is $\I$-convergent to $0$. (Just take the supremum of all group topologies with this property.)

If $\uu$ is a $T^\I$-sequence, then $T_\uu^\I$ is the supremum of all (not necessarily Hausdorff) group topologies on $G$ with the property that $\uu$ $\I$-converges to $0$. In fact, if $\sigma$ is a group topology on $G$ such that $u_n\overset{\I}{\longrightarrow}0$ in $(G,\sigma)$, then $u_n\overset{\I}{\longrightarrow}0$ also in $(G,T_\uu^\I\vee\sigma)$. Since $T_\uu^\I\vee\sigma$ is Hausdorff, $T_\uu^\I\vee\sigma=T_\uu^\I$, hence $\sigma\subseteq T_\uu^\I$.

Analogously, if $\uu$ is a $TB^\I$-sequence, then $T_{\bf pu}^\I$ is the finest totally bounded group topology on $G$ such that $\uu$ $\I$-converges to $0$.
\end{remark}

\begin{notation}\label{NotationBaicNbd}
Let $\I$ be free ideal on $\N_0$ and let $\uu$ be a a sequence in an abelian group $G$.
For $I\in\I$, let $$\uu_I=\{0\}\cup \{\pm u_n:\ n\notin I\}.$$
For $(I_n)_{n\in\N}$ a sequence of elements in ${\cal I}$, let
$$\uu_{(I_n)}=\bigcup_{N\in\N}(\uu_{I_1}+\ldots+\uu_{I_N}).$$
Let ${\cal I}^{\N}_{\uparrow}$ denote the family of all  increasing sequences in $\I$.
\end{notation}

\begin{remark} \label{41}
\begin{enumerate}
\item[(a)]
Let $\I$ be free ideal on $\N_0$ and $\uu$ a sequence in an abelian group $G$.
If $W$ is a symmetric neighborhood of $0$ in $T^{{\cal I}}_{{\bf u}} $, then there exists $I\in {\cal I}$ such that $\uu_I\subseteq W$.

\item[(b)] Recall that for every neighborhood $W_0$ of $0$ in a topological abelian group $(G,\tau)$, there exists a sequence $(W_n)$ of neighborhoods of $0$ in $\tau$ such that $W_{n+1}+W_{n+1}\subseteq W_n$ for all $n\in\N_0$ and hence $\bigcup_{n\in\N}(W_1+\ldots+W_n)\subseteq W_0$ holds.
\end{enumerate}
\end{remark}




\begin{proposition}\label{explicit}
Let $\I$ be a free ideal on $\N_0$ and let ${\bf u}$ be a $T^\I$-sequence in an abelian group $G$.
The family $\dis (\uu_{(I_n)})_{(I_n)\in {\cal I}^{\N}_{\uparrow}}$ forms a neighborhood base at $0$
for the finest (non-necessarily Hausdorff) group topology on $G$ in which $\uu$ $\I$-converges to $0$.
%
\end{proposition}
\begin{proof}
Let $(I_n),(J_n)\in\I^\N_{\uparrow}$. Then $(I_n\cup J_n)\in\I^\N_{\uparrow}$ and
$\uu_{(I_n\cup J_n)}\subseteq \uu_{(I_n)}\cap\uu_{(J_n)}.$
The set $\uu_{(I_n)}$ is symmetric and contains $0$.
 From the monotonicity
of $(\uu_{I_k})$ it follows that $(\uu_{I_2}+\ldots+ \uu_{I_{2n}})+(\uu_{I_2}+\ldots +\uu_{I_{2m}})\subseteq \uu_{I_1}+\ldots+\uu_{I_{2\max(m,n)}}\subseteq \uu_{(I_k)}$ for all $n,m\in\N$. This shows that
$\uu_{(I_{2n})}+\uu_{(I_{2n})}\subseteq \uu_{(I_n)}$. 

So far, we have shown that $(\uu_{(I_n)})_{(I_n)\in\I_\uparrow^\N}$ is a neighborhood base of a group topology $\eta$ on $G$. Next, we prove that $u_n\stackrel{{\cal I}}{\longrightarrow}0$ in $(G,\eta)$. To this end, fix a sequence $(I_k)$ in ${\cal I}^\N_\uparrow$. It is a consequence of
$$\{n\in \N:\ u_n\notin \uu_{(I_k)}\}\subseteq \{n\in \N:\ u_n\notin \uu_{I_1}\}\subseteq I_1\in {\cal I}$$
that $\{n\in \N:\ u_n\notin \uu_{(I_k)}\}\in {\cal I}$.

Conversely, let $\tau$ be any group topology on $G$ in which $\uu$ $\I$-converges to $0$. Fix a neighborhood $W_0\in \VV_\tau(0)$. According to Remark~\ref{41}(b), there exists a sequence $(W_n)$ of symmetric neighborhoods of $0$ such that $\bigcup_{N\in\N}(W_1+\ldots+W_N)\subseteq W$.
Since $\uu$ is $\I$-convergent to $0$ in $(G,\tau)$, for every $n\in\N$ the set $I_n=\{k\in\N:\  u_k\notin W_n\}\in\I$. In other words,
$\{u_k:\ k\notin I_n\}\subseteq W_n$ for all $n\in\N$. Since $W_n$ was assumed to be symmetric, we obtain
$\uu_{I_n}\subseteq W_n$. Let $J_n=\bigcup_{m\le n}I_m$; then $(J_n)\in\I^\N_\uparrow$ and, since $\uu_{J_n}\subseteq \uu_{I_n}$ for all $n\in\N$,  we may conclude that
$\uu_{(J_n)}\subseteq \bigcup_{N\in\N}(W_1+\ldots+W_N)\subseteq W_0$. This shows that $\eta$ is finer than $\tau$.
\end{proof}

\begin{corollary}
Let $\I$ be a free ideal on $\N_0$ and let $\uu$ be a $T^{\I}$-sequence in an abelian group $G$. Then $(\uu_{(I_n)})_{(I_n)\in\I^\N_\uparrow}$ is a neighborhood base at $0$ for $T^\I_{\uu}$.
\end{corollary}

\begin{remark}
For $\I = \FF in$, every increasing sequence $(I_n)$ in $\FF in^\N_\uparrow$ can be dominated by a sequence of the form $(\{1,\ldots,m_n\})$ where $(m_n)$ is an increasing sequence of natural numbers. 
Then
$$\uu_{(m_n)}=\uu_{(\{1,\ldots,m_n\})}=\bigcup_{N\in\N}(\uu_{m_1}+\ldots+\uu_{m_N})$$
and $(\uu_{(m_n)})$ where $(m_n)$ runs through all  increasing sequence in $\N $ forms a neighborhood base at $0$ for $T_{\uu}$. This is exactly the Protasov-Zelenyuk topology established in~\cite{PZ} (see also~\cite{Schar} and~\cite[\S 5.3.3]{ITG}).
\end{remark}

\begin{corollary}\label{CoroGlory} Let $\I$ be a free ideal on $\N_0$ and ${\bf u}$ a $T^\I$-sequence in an abelian group $G$. Then for every neighborhood $U$ of $0$ in $T_{{\bf u}}^{{\cal I}}$, there exists an
increasing sequence $(I_n)$ in $\I$ such that for every $N\in\N$, $N\odot \uu_{I_N}\subseteq U$. 
\end{corollary}
\begin{proof}
By Proposition~\ref{explicit}, there exists $(I_n)_{n\in\N}\in\I^\N_\uparrow$ such that $\uu_{(I_n)}=\bigcup_{N\in\N}(\uu_{I_1}+\ldots+\uu_{I_N})\subseteq U$. Hence, for every $N\in\N$, $N\odot \uu_{I_N}\subseteq \uu_{I_1}+\ldots+\uu_{I_N}\subseteq U$.
\end{proof}

Let $\I$ be a free ideal on $\N_0$. Call a sequence $\uu$ in an abelian group $G$ {\bf trivial} {\rm mod} $\I$, if $\{n\in \N: u_n \ne 0\}\in \I$. These are exactly the sequences that  $\I$-converge to 0 in the discrete topology on $G$.

\begin{corollary}
Let $\I$ be a free ideal on $\N_0$ and let $\uu$ be a sequence in an abelian group $G$. The following assertions are equivalent:
\begin{enumerate}
\item[(a)] $\uu$ is trivial {\rm mod} $\I$;
\item[(b)] $\uu$ is a $T^\I$-sequence and  $(G,T_{{\bf u}}^{{\cal I}})$ is discrete.
\end{enumerate}
\end{corollary}
\begin{proof}
The sequence $\uu$ is trivial {\rm mod} $\I$ if and only if $\uu$ $\I$-converges to $0$ in the discrete topology.
This is equivalent to $\uu$ is a $T^\I$-sequence and $T_{\uu}^{\I}$ is discrete.%
\end{proof}

This corollary suggests to consider mainly $T^\I$-sequences of non-zero elements.

\begin{corollary} \label{nhb_aP}
Let  ${\cal I}$ be a proper free $P$-ideal on $\N_0$ and $\uu$ a $T^\I$-sequence in an abelian group $G$.  Then there exists $I\in\I$ such that a neighborhood base at $0$ in $T_\uu^\I$ is given by the  family of sets 
$(\uu_{(I\cup F_n)})$ where $(F_n)$ is an increasing sequence of finite subsets of $\N$.
\end{corollary}
\begin{proof}
Let $(I_n)\in\I^\N_\uparrow$. Since ${\cal I}$ is a $P$-ideal, there exist $I\in {\cal I}$ and an increasing sequence $(F_n)$ of finite subsets of $\N_0$ (so all $F_n$ belong to ${\cal I}$) such that $I_n\subseteq I\cup F_n$. It follows directly from the definition of $\uu_{(I_n)}$ that
$\uu_{(I_n)}\supseteq \uu_{(I\cup F_n)}$ holds. Since $(I\cup F_n)\in {\cal I}^\N_\uparrow$, the assertion follows.
\end{proof}

\section{The connection of $T^\I$-sequences with their $T$-subsequences}\label{sub}

In this section, we shed light on the topology $T^\I_{\uu}$ of a $T^\I$-sequence $\uu$ by relating the topology $T^\I_{\uu}$ to topologies  $T_{\vv}$ for suitable subsequences $\vv$ of $\uu$, which are $T$-sequences. 

\begin{deff} \label{u_T}
For a sequence $\uu$ in an abelian group $G$, define the families
\begin{align*}
 \uu_T&=\{J\subseteq \N: \ \N\setminus J\ \mbox{is infinite and}\ (u_n)_{n\in \N \setminus J}\ \mbox{is a }T\mbox{-sequence}\},\\
 \uu_T'&=\{J\subseteq \N: \ J\ \mbox{is infinite and} \ (u_n)_{n\in J}\ \mbox{is a }T\mbox{-sequence}\},\\
 \uu_{TB}&=\{J\subseteq \N: \ \N\setminus J\ \mbox{is infinite and}\ (u_n)_{n\in \N \setminus J}\ \mbox{is a }TB\mbox{-sequence}\},\\
 \uu_{TB}'&=\{J\subseteq \N: \ J\ \mbox{is infinite and} \ (u_n)_{n\in J}\ \mbox{is a }TB\mbox{-sequence}\}.
\end{align*}
\end{deff}

Clearly, $\uu_T'$ (resp., $\uu_{TB}'$) is a subfamily of $[\N]^\omega$ (the family of all infinite subsets of $\N$),  that is stable under taking infinite subsets, while $\uu_T$ (resp., $\uu_{TB}$) consists of non-cofinite sets in $\N$,  and is stable under taking non-cofinite supersets.

 The families $\uu_T'$  and $\uu_{TB}'$  are not stable under finite unions: let $\uu=(u_n)$ be the sequence $u_{2k}=k!$ and $u_{2k+1}=k!+1$; then both $(u_{2n})_{n\in  \N}$ and $(u_{2n+1})_{n\in \N}$ are $TB$-sequences, but $(u_n)$ is not a $T$-sequence, since $(u_{2k+1}-u_{2k})=(1)$ does not converge to $0$.

Since there is an obvious bijection between $\uu_T$ and $\uu_T'$, they are simultaneously non-empty. Obviously, $\uu_T\ne \emptyset$ precisely when $\uu$ has a $T$-subsequence. Clearly,  $\uu_T = \emptyset = \uu_T'$  when the sets $ \{u_n:\ n\in\N\}$ and  $\{n\in\N:\ u_n=0\}$ are finite.

On the contrary, Theorem~\ref{classITth} implies that if the abelian group $G$ is either almost torsion free or of prime exponent, then every one-to-one sequence in $G$ (no matter if $\I$-convergent or not) has a subsequence which is a $TB$-sequence. Hence, for this class of abelian groups, the families $\uu_T$, $\uu'_T$, $\uu_{TB}$ and $\uu'_{TB}$ are always non-empty.

\begin{lemma}\label{ssc_Ic}
Let $\uu$ be a sequence in an abelian group $G$.
\begin{enumerate}
 \item[(a)]   If $\tau$ is a Hausdorff group topology on $G$ and $(u_n)_{n\in\N_0\setminus I}$ is  $\tau$-convergent for some $I\in {\cal I}$, then $\uu$ is a $T^\I$-sequence and the identity $\iota_I\colon(G,T^{{\cal I}}_{{\bf u}})\to (G,\tau)$ is continuous.  In particular, $T_{\uu}^{\I}\supseteq T_{(u_n)_{n\notin I}}$.
 \item[(b)]   Let $\I$ be a $P$-ideal. Then $\uu$ is  a $T^\I$-sequence if and only if there exists $I\in\I$ such that $(u_n)_{n\in\N_0\setminus I}$   is a $T$-sequence. In this situation $\iota_I\colon (G,T^{{\cal I}}_{{\bf u}})\to (G,T_{(u_n)_{n\in\N_0\setminus I}})$ is continuous.
\end{enumerate}
\end{lemma}
\begin{proof}
(a) The first assertion follows from Fact~\ref{Remark}(a) and the maximality of $T^\I_{{\bf u}}$, the second statement holds, since we may take $\tau=T_{(u_n)_{n\notin I}}$.

(b)  Assume now that $\uu$ is a  $T^\I$-sequence   and that ${\cal I}$ is a $P$-ideal. By Lemma
\ref{MetStat*}, there exists a metrizable group topology $\tau$ on $G$ coarser than $T^{{\cal I}}_{{\bf u}}$.  By Fact~\ref{Remark}(b), there exists $I\in \I$ such that $(u_n)_{n\notin I}$ converges to $0$ in $(G,\tau)$. Since $\tau$ is a Hausdorff group topology,
we conclude that $(u_n)_{n\notin I}$ is a $T$-sequence. The other implication as well as the additional statement follow  from (a)  applied to $T_{(u_n)_{n\notin I}}$.
\end{proof}

As a consequence of item (b) we obtain

\begin{example}\label{MetStat}
A sequence $\uu$ in a topological abelian group $(G,\tau)$ is statistically convergent to $0$ (i.e., a $T^{\I_d}$-sequence) if and only if there exists  $I\in \I_d$ such that $(u_n)_{n\in\N_0\setminus I}$ is convergent.
\end{example}

\begin{notation}
For a sequence ${\bf u}$ in an abelian group $G$ and a proper free ideal ${\cal I}$ on $\N_0$ we define
$${\cal I}^\uu_T=\{I\in {\cal I}:\ (u_n)_{n\in\N_0\setminus I}\ \mbox{is a }T\mbox{-sequence}\} = \I \cap \uu_T,$$
and analogously
$${\cal I}^\uu_{TB}=\{I\in {\cal I}:\ (u_n)_{n\in\N_0\setminus I}\ \mbox{is a }TB\mbox{-sequence}\} = \I \cap \uu_{TB}.$$
\end{notation}

\begin{proposition}
Let $\I$ be a proper free $P$-ideal on $\N_0$, let $G$ be an abelian group and $\uu$ a sequence in $G$.
Then $\uu$ is a $T^\I$-sequence if and only if $\I^\uu_T\ne \emptyset$.
\end{proposition}
\begin{proof}
By Lemma~\ref{ssc_Ic}, if $\I^\uu_T\ne \emptyset$, then $\uu$ is a $T^\I$-sequence.
By item (b) of the same lemma, if $\I$ is a $P$-ideal, then this implication can be inverted.
\end{proof}

  The next proposition characterizes $T$-sequences $\uu$ in terms of ${\cal I}^\uu_T$ being an ideal.

\begin{proposition}
Let $\I$ be a proper free ideal
 on $\N_0$ and $\uu$ a sequence in an abelian group $G$.
Then the following conditions are equivalent:
\begin{itemize}
\item[(a)]  ${\cal I}^\uu_T$ is an ideal;
\item[(b)] ${\cal I}^\uu_T = \I$;
\item[(c)] $\uu$ is a $T$-sequence.
\end{itemize}
\end{proposition}
\begin{proof}
(a)$\Rightarrow$(c) If $\I_T^{\uu}$ is an ideal, then $\emptyset \in \I^{\uu}_T$ and hence $\uu$ is a $T$-sequence.

(c)$\Rightarrow$(b) If $\uu$ is $T$-sequence, clearly $\I=\I_T^{\uu}$, since every subsequence of $\uu$ is a $T$-sequence as well.

(b)$\Rightarrow$(a) is obvious.
\end{proof}

  Next we see that when ${\cal I}_T^\uu$ is not empty, then ${\bf u}$ is a $T^\I$-sequence. The assumption that $\uu$ generates $G$ is not restrictive in view of Lemma~\ref{gen}.

\begin{proposition}\label{221}
Let ${\cal I}$ be a proper free ideal on $\N_0$, let $G$ be an abelian group and ${\bf u}$ be a sequence which generates $G$ and such that ${\cal I}_T^\uu$ is not empty.
Then ${\bf u}$ is a $T^\I$-sequence and for every $I\in {\cal I}^\uu_T$ the identity mapping $$\iota_I\colon(G,T^{{\cal I}}_{{\bf u}} )\to (G,T_{(u_n)_{n\in\N_0\setminus I}})$$ is continuous.


For a compact subset $C$ in $(G,T^{{\cal I}}_{{\bf u}} )$,
 $$C\subseteq \bigcap_{I\in \I_T}n_I\odot \{u_n:n\in\N_0\setminus I\}.$$
\end{proposition}
\begin{proof}
According to Lemma~\ref{ssc_Ic} the mapping $\iota_I$ is continuous for every $I\in {\cal I}_T$. So also the
 diagonal mapping is continuous. Fix a compact subset $C$ in $( G,T^{{\cal I}}_{{\bf u}} )$. Then $\iota_I(C)$ is a compact subset of $(G,T_{(u_n)_{n\in\N_0\setminus I}})$. According to Remark~\ref{hemicpt}, there exists $n_I\in\N$ such that
$\iota_I(C)\subseteq n_I\odot \{u_n:\ n\in \N_0\setminus I\}.$
\end{proof}

\begin{question}
Is it true that for any $T^\I$-sequence $\uu$ which spans an abelian group $G$ and such that $\I_T^\uu\not=\emptyset$, a subset $C$ of $G$ is compact if and only if there exists a family $(n_I)_{I\in\I_T}$ such that $C\subseteq \bigcap_{I\in \I_t}n_I\odot \{u_n: n\in\N_0\setminus I\}$?
\end{question}

Let $\uu$ be a $T$-sequence. 
The correspondence $\I \mapsto T^{{\cal I}}_{{\bf u}} $ is monotone increasing.
In particular, $T_\uu = T^{{\cal F}in}_{{\bf u}}  \subseteq T^{{\cal I}}_{{\bf u}} $. Hence, the $T^{{\cal I}}_{{\bf u}} $-compact sets are also
$T_\uu$-compact.

\smallskip
More in general, we leave the following questions open.

\begin{question}\label{Ques1}
Let $\I$ be a proper free ideal on $\N_0$ and $\uu$ a sequence in an abelian group $G$.
\begin{enumerate}
\item[(a)] 
What are the compact subsets of $(G,T^{{\cal I}}_{{\bf u}})$?

\item[(b)] Assume that $\I_T^{\uu}\not=\emptyset$. We proved that $T^{{\cal I}}_{{\bf u}} \supseteq \sup_{I \in \I_T^{\uu}} T_{(u_n)_{n\in\N_0\setminus I}} $. When do these two topologies coincide?
\end{enumerate}

\end{question}


\section{$\I$-characterized subgroups 
and ideal $T$- and $TB$-sequences} \label{susec}

\subsection{The character group of $(G,T^{{\cal I}}_{{\bf u}})$ and $(G,T^{{\cal I}}_{{\bf pu}})$}\hfill

The following is one of the main achievements of this paper. It generalizes to the ideal convergence setting the important equality known in the $\FF in$ case.

\begin{theorem}\label{cg}
Let $\I$ be a proper free ideal on $\N_0$ and $\uu$ a $T$-sequence in the abelian group $G$.
Algebraically, 
$$(G,T^{{\cal I}}_{{\bf u}})^\wedge={s}_{{\bf u}}^{{\cal I}}(G^\wedge).$$
\end{theorem}
\begin{proof}
Let $\chi\in (G,T^{{\cal I}}_{{\bf u}})^\wedge$. Then $\chi\in {\rm Hom}(G,\T)$ and $\chi(u_n)\stackrel{{\cal I}}{\longrightarrow}\chi(0)=0_\T$ by Fact~\ref{I-cont}.

Conversely, let $\chi\in {\rm Hom}(G,\T)$ such that $\chi(u_n)\stackrel{{\cal I}}{\longrightarrow}0$. Let $\sigma=\sigma(G,\langle\chi\rangle)$.
Then $u_n\stackrel{{\cal I}}{\longrightarrow}0$ in $(G,\sigma)$. Indeed, this condition is equivalent to:
for all $k\in\N$ the set $\{n\in\N:\ \chi(u_n)\notin \T_k\}\in {\cal I}$, but this means that $(\chi(u_n))$ ${\cal I}$-converges to $0$ in $\T$, which is true by hypothesis. Hence, 
$\sigma$ is coarser than $T^{{\cal I}}_{{\bf u}}$   by Remark~\ref{maximality}, which implies that $\chi\in (G,T^{{\cal I}}_{{\bf u}})^\wedge$.
\end{proof}

\begin{corollary}\label{cor1_cg}\label{cor2_cg}
Let $\I$ be a proper free ideal on $\N_0$ and let $\uu$ be a $T^\I$-sequence in an abelian group $G$. Then
\begin{equation}\label{*}
(G,T^{{\cal I}}_{{\bf u}} )^\wedge \supseteq\bigcup_{I\in {\cal I} }{s}_{(u_n)_{n\in\N_0\setminus I}}(G^\wedge).
\end{equation}
If $\I$ is a $P$-ideal, then \eqref{*} becomes an equality.
\end{corollary}
\begin{proof}
Fix $I\in {\cal I}$ and $\chi\in {s}_{(u_n)_{n\notin I}}(G^\wedge)$. This means that $(\chi(u_n))_{n\notin I}$ converges to $0_\T$.
According to Lemma~\ref{ssc_Ic},  $(\chi(u_n))_{n\in\N_0}$ is ${\cal I}$-convergent to $0_\T$.
Now the first assertion follows from Theorem~\ref{cg}.

To prove the second assertion, assume that $\I$ is also a $P$-ideal; it suffices to show that
$$\{\chi\in {\rm Hom}(G,\T):\chi(u_n)\stackrel{{\cal I}}{\longrightarrow}0_\T\}\subseteq \bigcup_{I\in {\cal I} }{s}_{(u_n)_{n\in\N_0\setminus I}}(G^\wedge).$$
Pick $\chi\in {\rm Hom}(G,\T)$ satisfying $\chi(u_n)\stackrel{{\cal I}}{\longrightarrow}0_\T$. Since $\T$ is first countable and ${\cal I}$ is a $P$-ideal, by Fact~\ref{Remark}(b) there exists $I\in {\cal I}$ such that $(\chi(u_n))_{n\notin I}$ is convergent to $0$. Hence, $\chi\in {s}_{(u_n)_{n\in\N_0\setminus I}}(G^\wedge)$.
\end{proof}


\begin{theorem}\label{tb}
Let $\I$ be a free ideal on $\N_0$, let $\uu$ be a sequence in an abelian group $G$, and let $K={\rm Hom}(G,\T)$ be the compact dual of $G$. Then the following holds:
\begin{enumerate}
\item[(a)]
$\sigma(G,s^{{\cal I}}_\uu(K))$ is the finest totally bounded group topology on $G$ in which $\uu$ ${\cal I}$-converges  to $0$.
\item[(b)] $\uu$ is a $TB^\I$-sequence if and only if $s^{{\cal I}}_\uu(K)$ is dense in $K$ (i.e., separates the points of $G$).
\end{enumerate}
 \end{theorem}
 \begin{proof}
(a) For every $\chi \in s^{{\cal I}}_\uu(K)$, $\uu$ ${\cal I}$-converges  to $0$ in the totally bounded group topology $\tau_\chi= \sigma(G,\langle \chi\rangle)$.
Since $\sigma(G,s^{{\cal I}}_\uu(K)) = \sup \{\tau_\chi: \chi \in s^{{\cal I}}_\uu(K))\}$, we deduce that $\uu$ ${\cal I}$-converges to $0$ in $\sigma(G,s^{{\cal I}}_\uu(K))$.
On the other hand, if $\sigma(G,H)$, with $H \leq K$, is an arbitrary totally bounded group topology on $G$ that makes $\uu$  ${\cal I}$-convergent  to $0$, then
for every $\chi \in H$ the continuity of $\chi\colon(G,\sigma(G,H))\to \T$ makes $(\chi(u_n))$
${\cal I}$-convergent  to $0$ in $\T$, so $\chi \in s^{{\cal I}}_\uu(K)$. This proves that $H \subseteq s^{{\cal I}}_\uu(K)$. Therefore, $\sigma(G,s^{{\cal I}}_\uu(K))$ is  the finest totally bounded group topology on $G$ in which $\uu$ ${\cal I}$-converges  to $0$.

(b) is a direct consequence of  (a).
\end{proof}

 As a consequence of Theorem~\ref{tb} and Remark~\ref{maximality}, we get that in case $\uu$ is a $TB^\I$-sequence, then $T_{\bf pu}^\I=\sigma(G,s^{{\cal I}}_\uu(G^\wedge))$.

\smallskip
If $(G,\tau)$ is a topological abelian group, $\tau^+$ denotes the totally bounded group topology $\sigma(G,(G,\tau)^\wedge)$. Of course, $\tau=\tau^+$
holds if and only if $\tau$ is totally bounded.

\begin{corollary}\label{1.1I}
Let $\I$ be a free ideal on $\N_0$ and let $\uu$ be a $TB^\I$-sequence in an abelian group $G$.
Then $$T^{{\cal I}}_{\mathbf p{\bf u}}(G) = T^{{\cal I}}_{{\bf u}}(G)^+=\sigma(G,s_\uu(G^\wedge))$$
and $s_{\uu}^\I(G^\wedge)=(G,T^\I_{\uu})^\wedge=(G,T^\I_{{\bf p}\uu})^\wedge$ holds.
\end{corollary}
\begin{proof}
Clearly, $T^\I_{\uu}\supseteq T^\I_{{\bf p}\uu}$ and hence
$(T^\I_{\uu})^+\supseteq (T^\I_{{\bf p}\uu})^+=T^\I_{{\bf p}\uu}$, where the last equality holds since $T^\I_{{\bf p}\uu}$ is precompact.
By the maximality of $T_{{\bf p}\uu}^\I$, since $T_{{\bf p}\uu}^\I\vee (T_{\uu}^\I)^+$ is totally bounded, we have $T_{{\bf p}\uu}^\I=T_{{\bf p}\uu}^\I\vee (T_{\uu}^\I)^+$ and hence $T_{{\bf p}\uu}^\I= (T_{\uu}^\I)^+$. By Theorem~\ref{cg}, $(G,T^\I_{\uu})^\wedge=s_{\uu}^\I(G^\wedge)$ holds, so the first assertion follows.
The second assertions follows from the first and the fact that $(G,\tau)^\wedge=(G,\tau^+)^\wedge$.
\end{proof}

\subsection{$\I$-characterized subgroups of compact abelian groups are $F_{\sigma\delta}$-sets}\hfill

For a topological abelian group $(G,\tau)$, the character group $G^\wedge$ can be considered as a subset of the compact group ${\rm Hom}(G,\T)$ (i.e., the character group of the discrete abelian group $G$). Since on compact groups, the Haar measure is available, it is of importance to know whether $G^\wedge$ is a measurable subset. In general, this is not true: take for example a non-measurable subgroup $D$ of $\T$ and consider the group
$(\Z,\sigma(\Z,D))$; of course, the character group $D$ is not measurable in ${\rm Hom}(\Z,\T)\cong \T$.
We will see now, that under some restrictive assumption on the ideal $\I$, for a compact abelian group $K$ and its discrete dual group $K^\wedge$ we have that the character group of $(K^\wedge,T^\I_{\uu})^\wedge$, i.e., the group $s_{\uu}^\I(K)$, is not only a measurable subgroup of $K$ but even a $F_{\sigma\delta}$ set.

First we establish the conditions which the ideal has to satisfy. Recall that an ideal $\I$ is a subset of the power set of $\N_0$ which can be identified with $\{0,1\}^{\N_0}$. An ideal is called {\bf analytic} if, considered as a subspace of the compact metric space $\{0,1\}^{\N_0}$ (with the product topology), it  is an analytic set, i.e., a continuous image of a Borel set.
Moreover, a function $\phi\colon\mathcal{P}(\N_0)\to [0,\infty]$ is called a {\bf submeasure} if
$\phi(\emptyset)=0$,
$\phi$ is monotone, (i.e., for all $A\subseteq B \subseteq \N_0$ one has $\phi(A)\le  \phi(B)$), and $\phi$ is
subadditive, (i.e., for all $A,B\subseteq \N_0$ one has $\phi(A\cup B)\le \phi(A)+\phi(B)$).
The submeasure $\phi$ is called {\bf finite} if $\phi(\N_0)<\infty$.

Let $X$ be a topological space; a function $f\colon X\to [0,\infty]$ is called {\bf lower semicontinuous} if for every sequence $(x_n)$ in $X$ converging to $x\in X$ it follows that $\liminf f(x_n)\ge f(x)$.
For a lower semicontinuous submeasure $\phi$ Solecki (\cite[p.52]{Sol}) defined the {\bf exhaustive ideal}
$${\rm Exh}(\phi)=\{A\subseteq \N_0:\ \lim_{n\to\infty}\phi(A\setminus \{0,1,\ldots,n\})=0\}.$$ It is straightforward to check that ${\rm Exh}(\phi)$ is an ideal; moreover it is an $F_{\sigma\delta}$-set and a $P$-ideal {\cite[Lemma~1.2.2]{Fa}}.
By~\cite[Theorem~3.1]{Sol}, every analytic $P$-ideal on $\N_0$ is of the form ${\rm Exh}(\phi)$ for a finite lower semicontinuous submeasure $\phi$.

First of all, $\FF in$ is an analytic free $P$-ideal.
According to~\cite[Proposition~3.25]{DRAH} (see also~\cite[Proposition~1.1]{BDFS}), also $\I_d$ is an analytic free $P$-ideal.
For other examples of analytic $P$-ideals see~\cite{Fa} (e.g., summable ideals).

\begin{theorem}\label{sd}
Let $\I$ be an analytic $P$-ideal on $\N_0$, $K$ a compact abelian group, and $\uu$ a sequence in the discrete abelian group $K^\wedge$. Then ${s}^{{\cal I}}_{{\bf u}}(K)$ is an $F_{\sigma\delta}$-set, in particular a Borel subset of $K$.
\end{theorem}
\begin{proof}
By~\cite[Theorem~3.1]{Sol}, ${\cal I}={\rm Exh}(\phi)$, where $\phi\colon{\cal P}(\N_0)\to [0,\infty[$ is a lower semicontinuous submeasure. Then
{\allowdisplaybreaks \begin{align*}
{\bf s}^{{\cal I}}_{{\bf u}}(K)
&=\{\chi\in K:\ \chi(u_n)\stackrel{{\cal I}}{\longrightarrow}0_\T\}=
\{\chi\in K:\ \forall k\in\N \ \{n\in\N_0: \chi(u_n)\notin \T_k\}\in {\cal I}\}=\\
&=\dis \bigcap_{k\in\N}\{\chi\in K:\{n\in\N_0: \chi(u_n)\notin \T_k\}\in {\cal I}\}=
\dis \bigcap_{k\in\N}\{\chi\in K:\{n\in\N_0: \chi(u_n)\notin \T_k\}\in {\rm Exh}(\phi)\}=\\
 &=\dis \bigcap_{k\in\N}\{\chi\in K:\lim_{j\to \infty}\phi(\{n\in\N_0: \chi(u_n)\notin \T_k\}\setminus [0,j])=0\}=\\
&=\dis \bigcap_{k\in\N}\left\{\chi\in K:\forall m\in\N\ \exists N\in\N\ \forall j\ge N \phi(\{n\in\N_0: \chi(u_n)\notin \T_k\}\setminus [0,j])\le \frac{1}{m}\right\}=\\
&=\dis \bigcap_{k\in\N} \bigcap_{m\in\N}\bigcup_{N\in\N}\bigcap_{j\ge N}\left \{\chi\in K:\phi(\{n>j: \chi(u_n)\notin \T_k\})\le \frac{1}{m}\right\}=\\
\stackrel{(\ast)}{=}&\dis \bigcap_{k\in\N} \bigcap_{m\in\N}\bigcup_{N\in\N}\bigcap_{j\ge N}\left\{\chi\in K: \forall \nu\ge j+1\ \phi(\{n\in \N_0: j< n\le \nu\ \wedge\  \chi(u_n)\notin \T_k\} )\le \frac{1}{m}\right\}=\\
&=\dis \bigcap_{k\in\N} \bigcap_{m\in\N}\bigcup_{N\in\N}\bigcap_{j\ge N}\bigcap_{\nu\ge j+1}\underbrace{\left\{\chi\in K:\phi(\{n\in \N_0:\ j+1\le n\le \nu\ \wedge\  \chi(u_n)\notin \T_k\})\le\frac{1}{m}\right\}}_{=:C_{k,m,j,\nu}};\\
\end{align*}}
here $(\ast)$ holds, since $\phi$ is lower semicontinuous and hence $\sigma$-subadditive (cf.,~\cite[Lemma 3.20]{DRAH}).

We are going to show that $C=C_{k,m,j,\nu}$ is closed.  Therefore, we look at the complement
$$K\setminus C= \left\{\chi\in K:\phi(\{n\in \N_0:\ j+1\le n\le \nu\ \wedge\  \chi(u_n)\notin \T_k\})>  \frac{1}{m}\right\}.$$
So fix $\chi\in K\setminus C$, and let
$$\{n_1<n_2<\ldots<n_\mu\}=\{n\in \N_0:\ j+1\le n\le \nu\ \wedge\  \chi(u_n)\notin \T_k\}.$$
Since $\T_k$ is closed,
there is a neighborhood $W$ of $\chi$ such that for all $\psi\in W$ we have
$\psi(u_{n_i})\notin \T_k$ for $1\le i\le \mu$. So, for all $\psi\in W$, we have
$$\{n_1,\ldots,n_\mu\}\subseteq \{n\in \N_0: j+1\le n\le \nu\ \wedge\  \psi(u_n)\notin \T_k\}$$ and hence $\phi(\{n\in \N_0:\ j+1\le n\le \nu\ \wedge\  \psi(u_n)\notin \T_k\})>\frac{1}{m}$. So $W\subseteq K\setminus C$, which shows that $C$ is closed.
\end{proof}

We do not know whether Theorem \ref{sd} remains true without the assumption that $\I$ is analytic. More precisely: 

\begin{question}\label{ques:sd}
(a) Do there exist a free $P$-ideal $\I$ on $\N_0$ and a sequence $\uu$ in $\N$ such that ${s}^{{\cal I}}_{{\bf u}}(\T)$ is not a Borel set? And with ${s}^{{\cal I}}_{{\bf u}}(\T)$ Borel but not $F_{\sigma\delta}$?

(b) Do there exist a compact abelian group $K$, a sequence $\uu$ in the discrete abelian group $K^\wedge$ and a free ($P$-ideal) $\I$ on $\N_0$ such that ${s}^{{\cal I}}_{{\bf u}}(K)$ is not an $F_{\sigma\delta}$-set 
of $K$? Could ${s}^{{\cal I}}_{{\bf u}}(K)$ be Borel but not $F_{\sigma\delta}$?

(c) For a maximal $P$-ideal $\I$ on $\N_0$, a compact abelian group $K$ and a sequence $\uu\in K^\wedge$, is ${s}^{{\cal I}}_{{\bf u}}(K)$ an $F_{\sigma\delta}$-set?
\end{question}

Note that a maximal free ideal on $\N_0$ cannot be analytic, as it is known that no non-principal ultrafilter can be analytic (see \cite{Medini} and \cite[Theorem~21.6]{Kechris}).


\begin{remark}
 Let $\I$ be an analytic $P$-ideal on $\N_0$, $K$ a compact abelian group with its Haar measure $\mu$, and $\uu$ a one-to-one sequence in $K^\wedge$.
By Theorem~\ref{sd}, $s_\uu(K)$ is a Borel set, so measurable.

Clearly, $\mu({s}^{{\cal I}}_{{\bf u}}(K))=0$ precisely when $s^{{\cal I}}_{{\bf u}}(K)$ has infinite index in $K$.
Moreover,~\cite[Lemma~1.2]{DK} shows that in case $H$ is a finite-index subgroup of $K$, then $H=s_\uu(K)$ where $\uu$ is a sequence such that each character in the annihilator of $H$ (which is finite) is listed infinitely many times.
\end{remark}

Now we reformulate Theorem~\ref{sd} in the setting that we used above in all the results of this paper, where $G$ is an abelian group and $\uu$ a sequence in $G$. In this case
we identify $G$ with the discrete dual of its compact dual $G^\wedge$, and write ${s}_{{\bf u}}(G^\wedge)$ for a sequence $\uu$ in $G$.

\begin{corollary}
Let $\I$ be an analytic $P$-ideal on $\N_0$ and $\uu$ a sequence in an abelian group $G$. Then $s^{{\cal I}}_{{\bf u}}(G^\wedge)$ is an $F_{\sigma\delta}$-set, in particular a Borel subset of $G^\wedge$.
\end{corollary}

\end{document}